\documentclass[11pt,reqno]{amsart}

\usepackage{amssymb,amsmath,amsthm,amscd,latexsym,amsfonts}
\usepackage{mathtools}
\usepackage{graphicx,epstopdf}
\usepackage{graphicx}
\usepackage{cite}

\usepackage[a4paper,top=3cm, bottom=3cm, left=3.1cm, right=3cm]{geometry}

\newtheorem{thm}{Theorem}
\newtheorem{defn}{Definition}
\newtheorem{lemma}{Lemma}
\newtheorem{pro}{Proposition}
\newtheorem{rk}{Remark}

\numberwithin{equation}{section} 

\newcommand{\M}{{\mathcal M}}

\newcommand{\bea}{\begin{eqnarray}}
\newcommand{\eea}{\end{eqnarray}}

\def\M{\mathcal M}

\begin{document}
\title [Quadratic non-stochastic operators]
{Quadratic non-stochastic operators: Examples of splitted chaos}

\author {U.A. Rozikov, \, S.S. Xudayarov}

\address{ U.Rozikov$^{a,b,c}$\begin{itemize}
 \item[$^a$] V.I.Romanovskiy Institute of Mathematics of Uzbek Academy of Sciences;
\item[$^b$] AKFA University, 1st Deadlock 10, Kukcha Darvoza, 100095, Tashkent, Uzbekistan;
\item[$^c$] Faculty of Mathematics, National University of Uzbekistan.
\end{itemize}}
\email{rozikovu@yandex.ru}

\address{S.\  \ Xudayarov\\ Bukhara State University, The department of Mathematics,
11, M.Iqbol, Bukhara city.\, Bukhara, Uzbekistan.}
\email {xsanat83@mail.ru}

\begin{abstract} There is one-to-one correspondence between quadratic operators (mapping  $\mathbb R^m$
to itself) and cubic matrices. It is known that any quadratic operator corresponding to a stochastic
(in a fixed sense) cubic matrix preserves the standard simplex.
In this paper we find conditions on the (non-stochastic) cubic matrix ensuring that corresponding
quadratic operator preserves simplex.  Moreover, we construct several quadratic non-stochastic operators which generate
chaotic dynamical systems on the simplex. These chaotic behaviors are \emph{splitted} meaning that the simplex is
partitioned into uncountably many invariant (with respect to quadratic operator)
subsets and the restriction of the dynamical system on each invariant set is chaos in the sense of Devaney.

\end{abstract}



\maketitle

\section{Introduction}

Non-linear dynamical systems arise in many problems of biology, physics and other sciences.
In particular, such dynamical systems describe the behavior of populations of different species with population models\footnote{https://en.wikipedia.org/wiki/Chaos$_-$theory} \cite{Rpd}.

It is known that there are some populations with regular behavior and other ones with chaotic behavior \cite{Ba}.
The chaos means sensitivity of behavior of the population to the tiniest changes in initial conditions
 (the initial state of the population) and unpredictable behavior.
 Mathematically studying a chaotic behavior is useful to our understanding of chaos as a phenomenon.
 In this paper we consider several quadratic (non-linear) mappings arising in population dynamics, which
 may generate a chaotic behavior. Let us give some basic notations:

{\bf Chaos.} For discrete-time dynamical systems a mathematical definition of chaos is as follows \cite{De}.
  Let $f$ be a function defined on some state space $X$.
  Denote $f^n(x)$, meaning $f$ is applied to $x\in X$ iteratively $n$ times.

   Furthermore, let $A$ be a subset of $X$. Then $f(A)=\{f(x): x\in A\}$.
   If $f(A)\subset A$, then $A$ is an invariant set under function $f$.

   A continuous map $f:X\to X$ is said to be topologically transitive if,
   for every pair of non-empty open sets $A,B\subset X$, there exists an integer $n$ such that
$f^{n}(A)\cap B\neq \varnothing.$

Devaney's definition (see\footnote{https://plato.stanford.edu/entries/chaos/} \cite{De}, \cite{H} for more details) of chaos is stated as follows:

A continuous map $f$ is chaotic if $f$ has an invariant set $A\subset X$ such that

1) $f$ satisfies weak sensitive dependence on its initial conditions on $A$,

2) The set of points initiating periodic orbits are dense in $A$,

3) $f$ is topologically transitive on $A$.\\

In \cite{Ban} it was observed that sensitive dependence on initial conditions
follows as a mathematical consequence of the other two properties.

Even simple processes can lead to chaos. This is reason why so hard to predict
the weather and the stock market.
One beautiful example is the game of billiards \cite{Rb}. Chaotic models are used in certain populations \cite{112} and in the population growth \cite{113}.
 Chaos can also be found in ecological systems, such as hydrology \cite{114}.
 Some biological application is found in cardiotocography. Models of warning signs of fetal hypoxia
 can be obtained through chaotic modeling \cite{115}.

{\bf Time evolution operators\footnote{https://en.wikipedia.org/wiki/Time$_-$evolution}.}
Time evolution is the change of state by the passage of time.
In general, time is not required to be a
continuous parameter, but may be discrete or even finite.

Consider a system with state space $X$ for which evolution is deterministic and reversible.
 For concreteness let us suppose time set $\mathbb T$ is the set $\mathbb R$ or $\mathbb N_0=\{0\}\cup \mathbb N$.

 Then time evolution is given by a family of state mappings
$$
F_{{t,s}}:X\rightarrow X,\quad \forall t,s\in {\mathbb T},$$
where $F_{t, s}(x)$ is the state of the system at time $t$, whose state at time $s$ is $x$. The following identity holds
\begin{equation}\label{F}
 {F}_{{u,t}}({F}_{{t,s}}(x))={F}_{{u,s}}(x).
\end{equation}
The mappings $F_{t, s}(x)$ are called evolution operators.

A state space with a distinguished evolution operators is called a dynamical system.

{\bf Examples}:

\emph{ 1. Markov process of square matrices.}
One of well studied time evolution is Markov process, which is defined by linear mappings as follows.
A family of stochastic matrices $\{ F_{s,t}=\mathcal U^{[s,t]}: s,t\geq 0\}$
is called a \emph{Markov process} if it satisfies the Kolmogorov-Chapman equation (i.e. equation (\ref{F})):
\begin{equation}\label{KC0}
\mathcal U^{[s,t]}=\mathcal U^{[s,\tau]}\mathcal U^{[\tau,t]}, \qquad \text{for all} \ \ 0\leq s<\tau<t.
\end{equation}
 Let $E=\{1,2,\dots,m\}$. A \emph{distribution} (or \emph{state}) of the set $E$  is a probability measure
 $x=(x_1,\dots,x_m)$, i.e. an element of the simplex:
\[S^{m-1}=\left\{x\in\mathbb R^m: x_i\geq 0, \ \sum_{i=1}^mx_i=1\right\}.\]

Let $x^{(0)}=(x_1^{(0)}, \dots, x_m^{(0)})\in S^{m-1}$ be an initial distribution on $E$.
Denote by $x^{(t)}=(x_1^{(t)}, \dots, x_m^{(t)})\in S^{m-1}$ the distribution of the system at the moment $t$.
For arbitrary moments of time $s$ and $t$ with $s<t$ the matrix $\mathcal U^{[s,t]}=\left(U^{[s,t]}_{ij}\right)$ gives the transition probabilities
from the distribution $x^{(s)}$ to the distribution $x^{(t)}$. Moreover $x^{(t)}$ depends linearly from $x^{(s)}$:
\[x^{(t)}_k=\sum_{i=1}^mU_{ik}^{[s,t]}x^{(s)}_i, \qquad k=1,\dots,m.\]

\emph{2. Quadratic stochastic process.}
Following \cite{LR} denote by $\mathcal S$ the set of all possible kinds of stochasticity and
denote by $\mathbb M$ the set
of all possible multiplication rules of cubic matrices.

Let
$\M^{[s,t]}=\left(P_{ijk}^{[s,t]}\right)_{i,j,k=1}^{\ \underset{m}{}}$ be a cubic matrix with two parameters.

 A family $\{F_{s,t}=\M^{[s,t]}: \ s,t\in\mathbb T\}$ is called
 (see \cite{LR}, \cite{MRX}, \cite{Rpd} for details) a Markov process of cubic matrices (or a quadratic stochastic process)
of type $(\sigma|\mu)$ if for each time $s$ and $t$ the cubic matrix
$\M^{[s,t]}$ is stochastic in sense $\sigma\in\mathcal S$ and satisfies the Kolmogorov-Chapman equation
(for cubic matrices):
\begin{equation}\label{KC}
\M^{[s,t]}=\M^{[s,\tau]}*_\mu\M^{[\tau,t]}, \qquad \text{for all} \ \ 0\leq s<\tau<t.
\end{equation} with respect to the multiplication $\mu\in \mathbb M$.

Quadratic stochastic processes arise naturally in the study of biological and
physical systems with interactions. Assume that the matrix $\left(P_{ijk}^{[s,t]}\right)$ is
3-stochastic (i.e., $P_{ijk}^{[s,t]}\geq 0$ and $\sum_kP_{ijk}^{[s,t]}=1$),
then the probability distribution $x^{(t)}$ (for the quadratic process)
can be found by the formula of the total probability as
\begin{equation}\label{Pi}
x^{(t)}_k= \sum_{i,j=1}^mP_{ijk}^{[s,t]}x^{(s)}_ix^{(s)}_j,
\end{equation}
where $k=1, \dots, m, \ \ 0\leq s<t$.

For case when $P_{ijk}^{[s,t]}$ does not depend on $s,t$ the theory of corresponding quadratic stochastic 
operator is well developed (\cite{GMR}, \cite{HR}, \cite{Ke}, \cite{ly}, \cite{MEN3} - \cite{Sad} and references therein).

\emph{3. Discrete-time quadratic dynamical systems.}
In this paper we consider discrete time, i.e.
$\mathbb T=\mathbb N_0=\{0\}\cup \mathbb N$ and
in the equality (\ref{Pi}) we assume the coefficients  $P_{ijk}^{[s,t]}$ do not depend
on $s$ and $t$, but the cubic matrix $\mathbb P=\left(P_{ijk}\right)$  is not assumed to be stochastic.

In general, a quadratic operator $V$, $V: x\in \mathbb R^m\to x'=V(x)\in \mathbb R^m$ corresponding to a cubic matrix $\mathbb P$ is defined by:
\begin{equation}\label{kv}
 V: x'_k=\sum \limits_{i,j=1}^{m} P_{ij,k} x_ix_j, \ \ k=1, \dots, m.
 \end{equation}

 Without loss of generality we assume  $P_{ij,k}= P_{ji,k}$. Indeed, if this equality is not satisfied then we can introduce
$$
\overline{P}_{ij,k}=\frac{1}{2}(P_{ij,k}+P_{ji,k}).
$$

The aim of this paper is to find conditions on the cubic matrix $\mathbb P$ ensuring that corresponding
operator preserves simplex $S^{m-1}$. Moreover, we want to construct several quadratic (non-stochastic) operators which generate
chaotic dynamical systems on the simplex.

\section{Quadratic non-stochastic operators}

The following theorem gives conditions for coefficients of $V$ to preserve the simplex.

\begin{thm}\label{tq1} For a quadratic operator $V$ (given by (\ref{kv})),  to preserve a simplex $S^{m-1}$ it is {\bf sufficient} that
\begin{itemize}
\item[i)] $\sum\limits_{k=1}^{m}P_{ij,k}=1, \,\,\ i,j=1,\dots, m;$

\item[ii)] $0\leq P_{ii,k}\leq1,\,\,\ i,k=1, \dots, m;$

\item[iii)] $-{1\over m-1}\sqrt{P_{ii,k}P_{jj,k}}\leq P_{ij,k}.$

and {\bf necessary} that the conditions i), ii) and
\item[iii')] $-\sqrt{P_{ii,k}P_{jj,k}}\leq P_{ij,k}\leq 1+\sqrt{(1-P_{ii,k})(1-P_{jj,k})}$
\end{itemize}
are satisfied.
\end{thm}
\begin{proof} {\bf Sufficiency.} Let $x\in S^{m-1}$. We show that $x'=V(x)=(x_1', \dots, x_m')\in S^{m-1}$.
Using condition i) we get
$$ \sum_{k=1}^mx'_k=\sum_{k=1}^m\sum \limits_{i,j=1}^{m} P_{ij,k} x_ix_j=\sum \limits_{i,j=1}^{m}\left(\sum_{k=1}^m P_{ij,k}\right) x_ix_j$$
$$=\sum \limits_{i,j=1}^{m} x_ix_j= \sum \limits_{i=1}^{m} x_i\sum \limits_{j=1}^{m}x_j=1, \ \ \forall k=1, \dots, m.$$
The condition ii) is needed to have $V(e_i)\in S^{m-1}$, for vertices $e_i=(0,...,0,1,0,...0)$ (here 1 is the $i$th coordinate).
Because in this case $V(e_i)=(P_{ii,1}, \dots, P_{ii,m}).$
Now using iii) we show that $x_k'\geq 0$:
$$ x'_k=\sum \limits_{i,j=1}^{m} P_{ij,k} x_ix_j=\sum_{i<j}\left({P_{ii,k}\over m-1}x_i^2+2P_{ij,k}x_ix_j+{P_{jj,k}\over m-1}x_j^2\right)$$
$$\geq \sum_{i<j}\left({P_{ii,k}\over m-1}x_i^2-{2\over m-1}\sqrt{P_{ii,k}P_{jj,k}}x_ix_j+{P_{jj,k}\over m-1}x_j^2\right)$$
$$=\sum_{i<j}\left(\sqrt{P_{ii,k}\over m-1}x_i-\sqrt{P_{jj,k}\over m-1}x_j\right)^2\geq 0.$$
Therefore the quadratic operator $V$ preserves the simplex.

\begin{rk} \begin{itemize}
\item[1.] In \cite{Sa}, for $m=2$,  it is proven that the conditions i), ii) and iii')
are sufficient and necessary to preserve the simplex.

\item[2.] In case $m\geq 3$ the conditions i), ii) and iii') are not sufficient to preserve
the simplex. Indeed, consider the example (satisfying  i), ii) and iii')):
$$\begin{array}{lll}
P_{ii,1}=1, \ \ P_{ii,k}=0, \ \ \forall i=1,\dots, m; \ \ \forall k=2,\dots, m;\\[2mm]
P_{ij,1}=- \sqrt{P_{ii,1}P_{jj,1}}=-1, \ \ \forall i\ne j;\\[2mm]
 P_{ij,k}\in [0, 2], \forall i\ne j, \, k\geq 2 \ \ \mbox{with} \ \ \sum\limits_{k=2}^{m}P_{ij,k}=2.
 \end{array}$$
Then for the first coordinate of the corresponding quadratic operator $V$ we have
$$x_1'=\sum_{i=1}^mx_i^2-2\sum_{1\leq i<j\leq m}x_ix_j=\left(x_1-\sum_{i=2}^mx_i\right)^2-4\sum_{2\leq i<j\leq m}x_ix_j$$
$$=(2x_1-1)^2-4\sum_{2\leq i<j\leq m}x_ix_j.$$
Take $x\in S^{m-1}$ such that $x_1=1/2$ and $x_j>0$, $j\geq 2$ with $\sum_{j=2}^mx_j =1/2$. Then $x_1'<0$, i.e. $x'=V(x)\notin S^{m-1}$.
\end{itemize}
\end{rk}
{\bf Necessity.}  Now following \cite{Sad} we prove necessity of conditions i)-iii').

Assume $V$ preserves the simplex.
As we mentioned above condition ii) is needed for $V(e_i)=(P_{ii,1}, \dots, P_{ii,m})\in S^{m-1}$, which also
requires $\sum_{k=1}^mP_{ii,k}=1$ (a particular case of the condition i)). To show that i) is necessary, let us take
$x=\alpha e_i+\beta e_j$, where $\alpha,\beta\geq 0$, $\alpha+\beta=1$.
Then for $V(x)$ we have
$$x_k'=P_{ii,k}\alpha^2+2P_{ij,k}\alpha\beta+P_{jj,k}\beta^2, \ \ k=1,\dots,m.$$
Since $V$  preserves simplex, we have
\begin{equation}\label{alpha}
0\leq P_{ii,k}\alpha^2+2P_{ij,k}\alpha\beta+P_{jj,k}\beta^2\leq 1, \ \ k=1,\dots,m.
\end{equation}
and since $\beta=1-\alpha$ we have
$$\sum_{k=1}^mx_k'=\sum_{k=1}^m\left(P_{ii,k}\alpha^2+2P_{ij,k}\alpha(1-\alpha)+P_{jj,k}(1-\alpha)^2\right)=1$$
which for $\alpha=1/2$ by $\sum_{k=1}^mP_{ii,k}=1$ gives i).

To obtain condition iii') in (\ref{alpha}) we denote
\begin{equation}\label{be}
a=P_{ii,k}, \ \ b=P_{ij,k}, \ \ c=P_{jj,k}.
\end{equation}
$$f(\alpha)=(a-2b+c)\alpha^2+2(b-c)\alpha+c.$$
Then inequalities (\ref{alpha}) equivalent to find conditions on
parameters $a,b,c$ such that $f(\alpha)\in [0,1]$ for each $\alpha\in [0,1]$.
We have $f(0)=c\in [0,1]$ and $f(1)=a\in [0,1]$.

{\it Case:} $a-2b+c=0$. That is $b=(a+c)/2\geq 0$. Then $f(\alpha)=2(b-c)\alpha+c$. The graph of this linear function
connects points $(0,c)$ and $(1,a)$ and contained in $[0,1]^2$.

{\it Case:} $a-2b+c\ne 0$. In this case the function $f(\alpha)$ has its extremum point
$$\alpha_0={c-b\over a-2b+c}.$$

Therefore, $f(\alpha)\in [0,1]$ if and only if one of the following conditions holds
\begin{itemize}
\item[(1)] $\alpha_0\notin (0,1)$
\item[(2)] $\alpha_0\in (0,1)$ and $f(\alpha_0)\in [0,1]$.
\end{itemize}
By elementary analysis it is easy to see that solution to the inequalities associated with conditions (1) and (2) is
$$-\sqrt{ac}\leq b\leq 1+\sqrt{(1-a)(1-c)}.$$
This by using notations (\ref{be}) shows that condition iii') is necessary too.
\end{proof}

\begin{defn} A quadratic operator (\ref{kv}), preserving a simplex, is called non-stochastic 
if at least one of its coefficients $P_{ij,k}$, $i\ne j$ is negative.
\end{defn}
\begin{defn} Let $V: S^{m-1}\rightarrow S^{m-1}$ be a quadratic non-stochastic operator (QnSO). It is called a Volterra QnSO (VQnVO) if

{\rm iv)} $P_{ij,k}=0,$ for all $k \notin \{i,j\}$.
\end{defn}

\begin{thm} If for a quadratic operator the conditions i)-iii) of Theorem \ref{tq1} are satisfied then it is not a VQnSO.
\end{thm}
\begin{proof}
Assume i)-iii) are satisfied then we get
$$P_{kk,k}=1, \ \ \ P_{ki,k}+P_{ik,k}=1.$$
$$ P_{ii,k}=0,\,\,\  i\neq k.  $$
Moreover, by  iii) we have

$$-{1\over m-1}\sqrt {P_{ii,k}}\sqrt {P_{kk,k}}\leq P_{ik,k} \quad \Rightarrow 0\leq P_{ik,k}, \,\ i\neq k.$$
Similarly we get $0\leq P_{ki,k}$. Thus all coefficients of Volterra quadratic operator are non-negative. Hence it is not VQnSO.
\end{proof}

\section{One-dimensional QnSO.}

Here we give a review of results related to the one-dimensional QnSO.
Consider arbitrary QnSO on $S^1$:
\begin{equation}\label{V1}
\begin{array}{ll}
x'=ax^2+2bxy+cy^2\\
y'=(1-a)x^2+2(1-b)xy+(1-c)y^2,
\end{array}
\end{equation}
where
\begin{equation}\label{par}
a,c\in [0,1], \ \ b\in [-\sqrt{ac}, \ \ 1+\sqrt{(1-a)(1-c)}].
\end{equation}

Using $x+y=1$, the operator (\ref{V1}) can be reduced to the function
$$f(x)=(a-2b+c)x^2+2(b-c)x+c.$$
Under condition (\ref{par}) we have $f:[0,1]\to [0,1]$. The dynamical system generated by $f$ can be fully
studied.
In \cite[Section 2.2]{MG}, the case $b\geq 0$ of this (stochastic) operator was studied.
In \cite{Sad} a class of QnSOs $V: S^1\to S^1$ were studied.

Here to avoid several cases, we consider the case $a=c=1$, then $b\in [-1,1]$.
The operator is QnSO iff $b\in [-1,0)$. Therefore, using $x+y=1$ from the second equality of (\ref{V1}) we get
$$y'=2(1-b)y(1-y), \ \ b\in [-1,0).$$
Denote $\mu=2(1-b)$. From  $b\in [-1,0)$ it follows that $\mu\in (2,4]$.

The function $g(y)\equiv g_\mu(y)=\mu y(1-y)$ is well-known as a logistic map.
For an initial point $x_0\in [0,1]$ consider the trajectory (dynamical system):
\begin{equation}\label{tr}
x_{n+1}=g(x_n), \ \ n=0,1,2,\dots
\end{equation}

For $\mu\in (2,4]$ this dynamical system has the following properties\footnote{https://en.wikipedia.org/wiki/Logistic$_-$map}:
\begin{rk}\label{ma}
\begin{itemize}
\item The function $g(y)$ has two fixed points $0$ and $1-{1\over \mu}$.

\item If $\mu$ between 2 and 3, the trajectory will eventually
approach the fixed point $1-{1\over \mu}$, but first will fluctuate around that value for some time.

\item If $\mu$ between 3 and $1 + \sqrt{6}\approx 3.44949$, from almost all initial point $x_0$
the trajectory will approach 2-periodic orbit. These two values are dependent on $\mu$.

\item With $\mu$ between 3.44949 and 3.54409 (approximately), from almost all initial point $x_0$ the trajectory will approach
 4-periodic orbit (permanent oscillations among four values).

\item With $\mu$ increasing beyond 3.54409, from almost all initial points the trajectory will approach
oscillations among 8 values, then 16, 32, etc.

\item At $\mu$ approximately equal to $3.56995$ from almost all initial points, we no longer see oscillations of
finite period. Slight variations in the initial point yield dramatically different results over time, a prime characteristic of chaos.

\item  A rough description of chaos is that chaotic systems exhibit a great sensitivity to initial points. The logistic map for most values of $\mu > 3.56995$ exhibits chaotic behavior.

\end{itemize}
\end{rk}
\begin{rk} We do not know any quadratic {\rm stochastic} operator with chaotic behavior of trajectories.
In the (above considered) case: $a=c=1$, $b\in [-1,0)$, the operator (\ref{V1}) has the form
\begin{equation}\label{Va}
\begin{array}{ll}
x'=x^2+2bxy+y^2\\
y'=2(1-b)xy.
\end{array}
\end{equation}
Since $P_{12,1}=b<0$, this operator is non-stochastic. For arbitrary initial point $(x_0, 1-x_0)\in S^1$, its trajectory has the form $(x_n, 1-x_n)$,
where $x_n$ is defined by (\ref{tr}). Therefore, from above-mentioned properties of the logistic map, it follows that when $-1\leq b< -0.784975$
the operator (\ref{Va}) generates a chaotic dynamical system on the one-dimensional simplex.
\end{rk}

\section{Examples of two-dimensional QnSO}

1. Consider the following example of QnSO on the two-dimensional simplex $S^2$:
\begin{equation}\label{V2}
\begin{array}{lll}
x'=x^2+y^2+z^2-axy-axz+2yz\\[2mm]
y'=(2+a)xy\\[2mm]
z'=(2+a)xz,
\end{array}
\end{equation}
where $a\in [0,2]$.
Note that $P_{12,1}=P_{13,1}=-a/2$.

1.1. {\it Fixed points}. The fixed points are solutions to the system
\begin{equation}\label{Vf}
\begin{array}{lll}
x=x^2+y^2+z^2-axy-axz+2yz\\[2mm]
y=(2+a)xy,\\[2mm]
z=(2+a)xz.
\end{array}
\end{equation}

In case $y=z=0$ we get the fixed point $p_0=(1,0,0)$.
For $y+z\ne 0$ we get $x={1\over 2+a}$. Consequently, the following is a family of fixed points
$$p_y=\left({1\over 2+a}, y, 1-{1\over 2+a}-y\right), \ \ {\rm where} \ \  y\in \left[0, 1-{1\over 2+a}\right].$$

1.2. {\it On invariant sets.} Recall that a  set $M$ is called invariant with respect to an operator $V$ if $V(M)\subset M$.

It is easy to see that the following sets are invariant with respect to (\ref{V2}):

$$M_0=\{(x,y,z)\in S^2: y=0\}, \ \   M_1=\{(x,y,z)\in S^2: z=0\},$$
$$M_\omega=\{(x,y,z)\in S^2: y=\omega z\}, \ \ \omega\in [0,+\infty).$$

Denoting $t=y+z$ we reduce the operator (\ref{V2}) to the following

\begin{equation}\label{VD}
\begin{array}{ll}
x'=x^2+t^2-axt\\[2mm]
t'=(2+a)xt.
\end{array}
\end{equation}
Since $x=1-t$. The operator (\ref{VD}) coincides with (\ref{Va}).
Therefore, under condition $1.56995<a\leq 2$ this operator and  (\ref{V2})
generate chaotic dynamical systems.
Moreover, using $y+z=t$ one can study trajectories of (\ref{V2}) by
related trajectories of (\ref{VD}). For example, if a trajectory $(x_n,t_n)$ of (\ref{VD}) has
a limit, say $(\alpha, \beta)$, then the corresponding trajectory $(x_n, y_n, z_n)$ has property that
$$\lim_{n\to\infty}x_n=\alpha, \ \ \lim_{n\to\infty} (y_n+z_n)=\beta.$$
The fixed point $t_*=1-{1\over 2+a}$ of the function $(2+a)t(1-t)$ gives an invariant set with $y+z=t_*$, i.e.,
$$X=\{(x,y,z)\in S^2: x={1\over 2+a}\}.$$
Note that the above mentioned fixed point $p_y$ is the following
$$p_y=X\cap M_\omega, \ \ {\rm where} \ \ \omega={(2+a)y\over 1+a-(2+a)y}.$$
Moreover,
$$S^2=M_0\cup M_1\cup (\bigcup_{\omega\in (0,+\infty)}M_\omega).$$
Therefore, it suffices to study restrictions of the operator (\ref{V2}) on each these invariant sets.

Restriction of the operator (\ref{V2}) on $M_0$ (using $x=1-z$) can be written as the function
$z'=(2+a)z(1-z)$. Similarly, on $M_1$ one has the same function $y'=(2+a)y(1-y).$

For each $\omega\in (0,+\infty)$, on the set $M_\omega$ the restriction of (\ref{V2}) can be written as
\begin{equation}\label{w}
z'=(2+a)(1-(\omega+1)z)z.
\end{equation}
Multiply both side of (\ref{w}) to $1+\omega$ and denote $\zeta=(1+\omega)z$ then we get
$\zeta'=(2+a)\zeta(1-\zeta).$

Therefore, the trajectories of the operator on the invariants $M_0$, $M_1$ and $M_\omega$ are given by the same logistic function.
For $\mu=2+a$ using facts of Remark \ref{ma} one can give dynamics of these functions
(i.e., the dynamics of operator (\ref{V2}) on the above mentioned invariants). In particular, under condition $1.56995<a\leq 2$
each one-dimensional dynamical system is chaotic.

\begin{rk}
As we have seen, the operator (\ref{V2}) is chaotic for $1.56995<a\leq 2$, but it is not
chaos on the simplex in the sense of Devaney. Because, it is not topologically transitive.
It is \emph{splitted} chaos meaning that the simplex is
partitioned into uncountably many invariant
subsets and the restriction of the operator (\ref{V2}) on each invariant set is chaos in the sense of Devaney.
\end{rk}

2. Consider the following example:
$$\begin{array}{lll}
P_{ii,1}=1, \ \ P_{ii,k}=0, \ \ \forall i=1,2,3; \ \ \forall k=2,3;\\[2mm]
P_{ij,1}=-{1\over 2} \sqrt{P_{ii,1}P_{jj,1}}=-{1\over 2}, \ \ \forall i\ne j;\\[2mm]
 P_{ij,k}\in [0, {3\over 2}], \forall i\ne j, \, k=2,3 \ \ \mbox{with} \ \ P_{ij,2}+P_{ij,3}={3\over 2}.
 \end{array}$$
Then taking some parameters equal to zero we get the following quadratic operator $V$:
\begin{equation}\label{V3}
\begin{array}{lll}
x'=x^2+y^2+z^2-xy-xz-yz\\[2mm]
y'=3xy+ayz\\[2mm]
z'=3xz+(3-a)yz,
\end{array}
\end{equation}
where $a\in (0,3)$.

\begin{rk} In the operator (\ref{V3}) one can also consider the cases $a=0$ and $a=3$.
These cases are more simple than the case $a\in (0,3)$. Because, for example, if $a=0$ then
$z'=3z(x+y)=3z(1-z)$, i.e., the variable $z$ has dynamics independent from other variables.
Therefore, below we consider the case $a\ne 0$, $a\ne 3$.

\end{rk}

2.1. {\it Fixed points}. It is easy to see that the fixed points of the operator (\ref{V3}) are
\begin{equation}\label{ss}
\begin{split}
s_1=(1,0,0), s_2=\left({1\over 3}, 0, {2\over 3}\right), s_3=\left({1\over 3}, {2\over 3}, 0\right),\\[3mm]
s_4=\left({a^2-3a+3\over a^2-3a+9}, {2a\over a^2-3a+9}, {2(3-a)\over a^2-3a+9} \right).
\end{split}
\end{equation}

\begin{defn}\cite{De}. A fixed point $x^*$ of the operator  $V$ is called hyperbolic if its
Jacobian $J$  at $x^*$ has no eigenvalues on the unit circle.
\end{defn}

\begin{defn} \cite{De}. A hyperbolic fixed point $x^*$ is called:
\begin{itemize}
  \item [i)]  attracting  if all the eigenvalues of the Jacobian $J(x^*)$ are less than 1 in absolute value;
  \item [ii)] repelling   if all the eigenvalues of the Jacobian $J(x^*)$ are greater than 1 in absolute value;
 \item [iii)] a saddle    otherwise.
 \end{itemize}
\end{defn}

 To study the type of each fixed point rewrite operator (\ref{V3}) (using $x=1-y-z$) as
 \begin{equation}\label{W}
W: \begin{array}{ll}
y'=3xy+ayz=y(3-3y+(a-3)z)\\[2mm]
z'=3xz+(3-a)yz=z(3-ay-3z).
\end{array}
\end{equation}
Note that $W$ maps the set $T=\{(y,z)\in [0,1]^2: y+z\leq 1\}$ to itself.

The Jacobian of $W$ at point $(y,z)$ is
$$J_W(y,z)=\left(\begin{array}{cc}
3-6y+(a-3)z & (a-3)y\\[2mm]
 -az & 3-ay-6z
 \end{array}\right).$$

 For eigenvalues of the Jacobian at fixed points we have
 \begin{itemize}
 \item[Case $s_1$:]
 $$\lambda_1=\lambda_2=3$$
 \item[Case $s_2$:]
   $$\lambda_1=1+{2a\over 3}\in[1,3], \ \ \lambda_2=-1$$
 \item[Case $s_3$:]
   $$\lambda_1=-1, \ \ \lambda_2=3-{2a\over 3}\in [1,3]$$
 \item[Case $s_4$:] In this case $\lambda_1$ has a bulky form. But using Maple one can plot its graph
 (see Fig. \ref{ev1}). Therefore, $0<\lambda_1<1$ for any $a\in (0,3)$. Moreover, one can see that $\lambda_2=-1$.
   \end{itemize}
 \begin{figure}
\includegraphics[width=6cm]{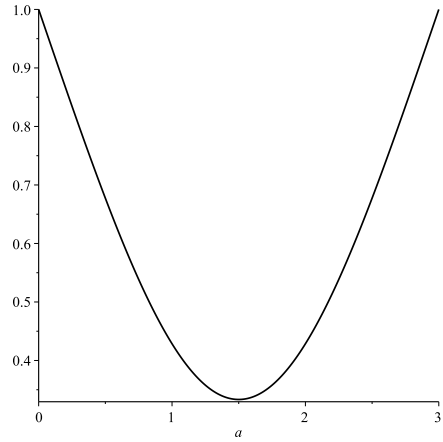}
\caption{The graph of the eigenvalue $\lambda_1(s_4)$ as function of parameter $a\in [0,3]$.}\label{ev1}
\end{figure}

Thus we have proved the following
\begin{pro} Fixed point $s_1$ is repeller. Points $s_2$ and $s_3$ are non hyperbolic (but semi-repeller\footnote{meaning that the second eigenvalue is greater than 1 in absolute value.}).
The fixed point $s_4$ is non-hyperbolic (but semi-attracting\footnote{meaning that the second eigenvalue is less than 1 in absolute value.}).
\end{pro}

2.2. {\it Invariant sets.} Let $a\in (0,3)$. Introduce the following sets:
$$M_1=\{(y,z)\in T: y=0\}, \ \ M_2= \{(y,z)\in T: z=0\},$$
$$M_3=\{(y,z)\in T: z={3-a\over a}y\}.$$
$$M_4=\{(y,z)\in T: z<{3-a\over a}y\}, \ \ M_5=\{(y,z)\in T: z>{3-a\over a}y\}.$$
\begin{lemma} The sets $M_i$, $i=1, \dots, 5$ are invariant with respect to the operator $W$, i.e. (\ref{W}).
\end{lemma}
\begin{proof} It is easy to see that $W(M_i)\subset M_i$, $i=1,2$. For the case $i=3,4,5$ assume $(y,z)\in M_i$,
we shall show that $(y',z')=W(y,z)\in M_i$. From the first and second equalities of (\ref{W}) we find
$$yz={1\over a}(y'-3xy), \ \ \  yz={1\over 3-a}(z'-3xz).$$
Consequently
\begin{equation}\label{in}{1\over a}(y'-3xy)={1\over 3-a}(z'-3xz)\ \ \Leftrightarrow \ \  (3-a)y'-az'=3x((3-a)y-az).
\end{equation}
From the last equality it follows that
$$(3-a)y-az=0 \ \ \Leftrightarrow \ \ (3-a)y'-az'=0.$$

Thus $M_3$ is an invariant. Moreover, from (\ref{in}) it follows
that if $x=0$ then $(3-a)y'-az'=0$, this means that any point $(y,z)$ with $y+z=1$ (i.e. $x=0$)
after first iteration goes inside of the invariant set $M_3$. Therefore, consider $x>0$, then
from (\ref{in}) we get
$$(3-a)y-az>0 \ \ \Leftrightarrow \ \ (3-a)y'-az'> 0.$$
$$(3-a)y-az<0 \ \ \Leftrightarrow \ \ (3-a)y'-az'< 0.$$
Thus $M_4$ and $M_5$ are invariant sets.
\end{proof}
Note that
\begin{equation}\label{n}
T=\bigcup_{i=1}^5M_i.
\end{equation}

2.3. {\it Trajectories.} In this subsection for any initial point $(y^{(0)}, z^{(0)})\in T$ we investigate behavior
of the trajectories $(y^{(n)}, z^{(n)})=W^n(y^{(0)}, z^{(0)}), \ \ n\geq 1.$

By (\ref{n}) it suffices to study the trajectories on each invariant set.\\

{\it Case $M_1$:} Reducing (\ref{W}) on $M_1$ we get one-dimensional dynamical
system generated by $z'=3z(1-z)$ which is a logistic map with parameter $\mu=3$.
For this function it is known (see Remark \ref{ma} and \cite[page 10]{SKS}) that it has repeller fixed point $z=0$ and
attracting fixed point $z=2/3$. Consequently, for trajectory of the operator (\ref{W}) on the invariant set $M_1$ we have
$$\lim_{n\to\infty}(y^{(n)}, z^{(n)})=\lim_{n\to\infty}W^n(y^{(0)}, z^{(0)})=\left\{\begin{array}{ll}
(0,0), \ \ \mbox{if} \ \ (y^{(0)}, z^{(0)})=(0,0)\\[2mm]
(0,2/3), \ \ \mbox{if} \ \ (y^{(0)}, z^{(0)})=(0,z^{(0)}), z^{(0)}>0.
\end{array}\right.$$

{\it Case $M_2$:} is similar to the case $M_1$.\\

{\it Case $M_3$:} Restricting $W$ on $M_3$ we get $y'=3y(1-{9-3a+a^2\over 3a}y)$. Denoting $t={9-3a+a^2\over 3a}y$
and $t'={9-3a+a^2\over 3a}y'$ the last mapping can be written as $t'=3t(1-t)$. Therefore, this case also conjugate to
the above cases and the following holds

$$\lim_{n\to\infty}(y^{(n)}, z^{(n)})=\left\{\begin{array}{ll}
(0,0), \ \ \mbox{if} \ \ (y^{(0)}, z^{(0)})=(0,0)\\[2mm]
\left({2a\over a^2-3a+9}, {2(3-a)\over a^2-3a+9} \right), \ \ \mbox{if} \ \ (y^{(0)}, z^{(0)})=(y^{(0)}, {3-a\over a}y^{(0)}), y^{(0)}>0.
\end{array}\right.$$

{\it Case $M_4$:} In this case for $y>0$ we have ${z\over y}<{3-a\over a}$ and
${z'\over y'}<{3-a\over a}$. Iterating the inequalities we get
\begin{equation}\label{p}
P_n:={z^{(n)}\over y^{(n)}}<{3-a\over a}, \ n\geq 1.
\end{equation}
By (\ref{W}) we have
$$P_{n+1}=P_n\cdot {3-ay^{(n)}-3z^{(n)}\over 3-3y^{(n)}+(a-3)z^{(n)}}>P_n,$$
because:
$${3-ay^{(n)}-3z^{(n)}\over 3-3y^{(n)}+(a-3)z^{(n)}}> 1\ \ \Leftrightarrow \ \ -ay^{(n)}-3z^{(n)}>-3y^{(n)}+(a-3)z^{(n)} \ \ \Leftrightarrow \ \ P_n<{3-a\over a}.$$
Hence $P_n$ is strictly increasing and with the upper bound ${3-a\over a}$. Consequently, its limit exists and equal to
${3-a\over a}$ as the supremum of $P_n$ in $M_3\cup M_4$.

{\it Case $M_5$:} This case is similar to the case $M_4$, now $P_n>{3-a\over a}$
and $P_n$ is strictly decreasing with limit ${3-a\over a}$ too.

Thus for any initial point from $M_4\cup M_5$ the set of limit points of its trajectory is subset of $M_3$.

Now we give limit points
of the operator (\ref{V3}). To do this we introduce:
$$\hat M_i=\{(x,y,z)\in S^2: (y,z)\in M_i\}, \ \ i=1,2,3,4,5.$$
Then
$$S^2=\bigcup_{i=1}^5\hat M_i.$$

Summarizing above-mentioned results about trajectories of $W$,  we obtain the following

\begin{thm} If $(x^{(0)}, y^{(0)}, z^{(0)})\in \hat M_i$ for some $i=1,2,3,4,5$ then for the operator (\ref{V3})
the following holds
$$\lim_{n\to \infty}V^n(x^{(0)}, y^{(0)}, z^{(0)})=\left\{\begin{array}{lllll}
s_1 \ \ \mbox{if} \ \ x^{(0)}=1\\[2mm]
s_2 \ \ \mbox{if} \ \ i=1, z^{(0)}>0\\[2mm]
s_3 \ \ \mbox{if} \ \ i=2, y^{(0)}>0\\[2mm]
s_4 \ \ \mbox{if} \ \ i=3, y^{(0)}>0\\[2mm]
\in \hat M_3 \ \ \mbox{if} \ \ i=4, 5,
\end{array}\right.$$
where $s_i$, $i=1,2,3,4$ are defined in (\ref{ss}).
\end{thm}
Based on numerical analysis we make the following

{\bf Conjecture.} If $(x^{(0)}, y^{(0)}, z^{(0)})\in \hat M_4\cup \hat M_5$ then for the operator (\ref{V3})
the following holds
$$\lim_{n\to \infty}V^n(x^{(0)}, y^{(0)}, z^{(0)})=s_4.$$

Thus the operator (\ref{V3}) does not generate a chaotic dynamical system.
%


\end{document}